\documentclass[12pt,reqno]{amsart}
\usepackage{amssymb, amsfonts,amsmath,amsthm}
\usepackage{latexsym}
\usepackage{textcomp}
\usepackage{amsaddr}

\newtheorem{thm}{Theorem}[section]
\newtheorem{prop}[thm]{Proposition}

\newtheorem{lem}[thm]{Lemma}
\newtheorem{cor}[thm]{Corollary}

\newtheorem{dfn}[thm]{Definition}
\newtheorem{rmk}[thm]{Remark}
\newtheorem{qs}[thm]{Question}

\DeclareMathOperator{\GL}{GL}
\DeclareMathOperator{\aut}{Aut}
\DeclareMathOperator{\End}{End}
\DeclareMathOperator{\Out}{Out}

\newcommand{\inv}{^{-1}}

\title[Gap between a free group automorphism and its inverse]{Bounding the gap between a free group (outer) automorphism and its inverse}

\author{M.~Ladra}
\address{Department of Algebra,  University of Santiago de Compostela,\\15782
Santiago de Compostela, Spain}
\email{manuel.ladra@usc.es}

 \author{P.V.~Silva}
\address{Centro de Matem\'{a}tica, Faculdade de Ci\^{e}ncias,
Universidade do Porto, \\
R. Campo Alegre 687, 4169-007 Porto, Portugal}
\email{pvsilva@fc.up.pt}

\author{E.~Ventura}
\address{Dept. Mat. Apl.
III, Universitat Polit\`ecnica de Catalunya,\\Manresa, Barcelona, Catalunya}
\email{enric.ventura@upc.edu}

\begin{document}

\begin{abstract}
For any finitely generated group $G$, two complexity functions $\alpha_G$ and $\beta_G$
are defined to measure the maximal possible gap between the norm of an automorphism
(respectively outer automorphism) of $G$ and the norm of its inverse. Restricting
attention to free groups, $F_r$, the exact asymptotic behaviour of $\alpha_2$ and
$\beta_2$ is computed. For rank $r\geqslant 3$, polynomial lower bounds are provided for
$\alpha_r$ and $\beta_r$, and the existence of a polynomial upper bound is proved for
$\beta_r$.
\end{abstract}


\maketitle

\section{Introduction}

The goal of this paper is to study automorphisms of groups, specifically to introduce a new
technique to measure how easy or difficult is it to invert them. With this in mind, we
associate two new functions, $\alpha_G(n)$ and $\beta_G(n)$, to the group $G$ and propose
to study its asymptotic behaviour.

In the present introduction we define these functions in general, and show they are
independent from the set of generators, up to multiplicative constants. Then, for the rest
of the paper, we restrict our attention to finitely generated free groups and give several
results concerning the asymptotic growth of their corresponding functions. A similar
project can be carried out in any other families of groups $G$; we hope the study of these
new functions motivates new interesting results in the near future.

Let $G$ be a finitely generated group, and let us fix a finite set of generators $A=\{
a_1,\, \ldots ,a_r\}$.

This naturally gives a metric on $G$: every element $g\in G$ can be written as a product of
the $a_i$'s and their inverses, and one defines $|g|_A$ to be the length of the shortest
such expression i.e., $|g|_A\leqslant n$ if and only if $g=a_{i_1}^{\epsilon_1}\cdots
a_{i_m}^{\epsilon_m}$ for some $m\leqslant n$, some indices $i_1,\ldots ,i_m\in \{ 1,\ldots
,r\}$ and some signs $\epsilon_i =\pm 1$. Of course, $|1|_A =0$, $|g^n|_A \leqslant
|n||g|_A$, and $|gg'|_A \leqslant |g|_A +|g'|_A$ hold for all $g,g' \in G$ and all integer
$n$.

The same can be done with an infinite set of generators. However, $|A|<\infty$ gives us
finiteness of balls, $|\{ g\in G \mid |g|_A \leqslant n\}|<\infty$, which is a crucial
property in many respects; for example, in our definitions below.

Let us consider the group of automorphisms of $G$, $\aut G$. We let automorphisms act on
the right, so we write $\varphi \colon G\to G$, $g\mapsto g\varphi$. For every $g\in G$, we
denote by $\lambda_g$ the right conjugation by $g$, namely $x\lambda_g = g^{-1} xg$. Since
$\lambda_g\varphi = \varphi\lambda_{g\varphi}$, it follows easily that $\Lambda = \{
\lambda_g \mid g\in G\}$ is a normal subgroup of $\aut G$. Each of the cosets $[\varphi] =
\varphi\Lambda$ is said to be an \emph{outer automorphism} of $G$. We write $\Out G=(\aut
G)/\Lambda$.

Of course, every automorphism $\varphi \in \aut G$ is determined by the images of the
generators $a_1,\, \ldots ,\, a_r$. And the sum of its lengths is a good measure of the
complexity of $\varphi$ (understood as a rule moving elements of $G$ around). Let us define
then the \emph{norm} of $\varphi$ as
 $$
\|\varphi \|_A = |a_1\varphi |_A +\cdots +|a_r\varphi |_A.
 $$

Note that there is no $\varphi \in \aut G$ with $\| \varphi\|_A\leqslant r-1$, because
$a_i\varphi \neq 1$ for all $i$; the shortest automorphism (among possibly others) is the
identity, $\| Id_G \|_A=r$. Note also that, for increasing values of $n\geqslant r$, there
is a non-decreasing number of automorphisms $\varphi \in \aut G$ with $\| \varphi \|_A
\leqslant n$, but only finitely many for every fixed $n$. Observe also that $\| g\varphi
\|_A \leqslant |g|_A \cdot \| \varphi\|_A$ for all $g\in G$ and all $\varphi \in \aut G$.

This measure induces a similar measure on $\Out G$, defined as follows. Given $\Phi \in
\Out G$, we define the \emph{norm} of $\Phi$ as
 $$
\|\Phi \|_A = \min \, \{ \|\varphi \|_A \mid \varphi \in \Phi \}.
 $$
Once again, for every fixed $n$, there exists a finite number of outer automorphisms $\Phi
\in \Out G$ with $\| \Phi \|_A \leqslant n$.

A natural question is to ask about the relation between $\|\varphi \|_A$ and
$\|\varphi^{-1}\|_A$ (resp., between $\|\Phi \|_A$ and $\|\Phi^{-1}\|_A$). If one
happens to be significantly bigger than the other, then it intuitively means that
inverting such an automorphism is hard (just writing down the expression of
$\varphi^{-1}$ as images of the generators will take much longer than doing the same
for $\varphi$). With the purpose of measuring the (worst case) difference between the
complexity of an automorphism $\varphi$ and that of $\varphi^{-1}$, we define the
following complexity functions $\alpha_A,\, \beta_A \colon \mathbb{N}\to \mathbb{N}$,
\begin{align*}
\alpha_A(n) & = \  \max \, \{ \| \varphi^{-1}\|_A \mid \varphi \in \aut G,\quad
\|\varphi\|_A \leqslant n\}, \\[+.2cm] \beta_A(n) & = \  \max \, \{ \| \Phi^{-1}\|_A \mid
\Phi \in \Out G,\quad \|\Phi\|_A \leqslant n\},
\end{align*}
where, by convention, we take $\max \emptyset =0$ (i.e., $\alpha_A(n)=\beta_A (n)=0$
for $n=0,1,\ldots ,r-1$).

Clearly, $\alpha_A(n)\leqslant \alpha_A(n+1)$ and $\beta_A(n)\leqslant \beta_A(n+1)$ that
is, $\alpha_A$ and $\beta_A$ are non-decreasing functions. Furthermore, it is immediate
that $\beta_A(n) = \max \, \{ \| [\varphi^{-1}] \|_A \mid \varphi \in \aut G,\quad
\|\varphi\|_A \leqslant n\}$, hence $\beta_A(n) \leqslant \alpha_A(n)$ for every
$n\geqslant 0$.

As we have emphasized in the notation, the values of $|g|_A$, $\| \varphi \|_A$ and $\|
[\varphi ]\|_A$, as well as the functions $\alpha_A$ and $\beta_A$, do depend on the
preselected generating set $A$. However, the asymptotic behaviour of these last two
functions do not depend on $A$ and so, they will constitute two invariants of the group
$G$. More precisely, changing to another finite generating system these two functions
change only up to multiplicative constants both at the domain and at the range, as proved
in the following proposition.

\begin{lem}
Let $G$ be a group, and let $A=\{ a_1,\ldots ,a_r\}$ and $B=\{ b_1,\ldots ,b_s\}$ be two
finite generating sets. Then, there exists a constant $C\geqslant 1$ such that, for all
$\varphi \in \aut G$ and $\Phi \in \Out G$, the following inequalities hold:
\begin{itemize}
\item[(i)] $\frac{1}{C}\| \varphi\|_B \leqslant \| \varphi \|_A \leqslant C\| \varphi
    \|_B$,
\item[(ii)] $\frac{1}{C}\| \Phi \|_B \leqslant \| \Phi \|_A \leqslant C\| \Phi \|_B$.
\end{itemize}
\end{lem}

\begin{proof}
Take $M=\max \{ |b_i|_A \mid i=1,\ldots ,s\}$, $N=\max \{ |a_i|_B \mid i=1,\ldots ,r\}$,
and let $C=MNrs\geqslant 1$. For every $\varphi \in \aut G$ we have
 $$
\begin{array}{rcl}
\| \varphi \|_B & = & |b_1\varphi|_B+\cdots +|b_s\varphi |_B \\ & \leqslant & |b_1\varphi
|_A N+ \cdots +|b_s\varphi |_A N \\ & \leqslant & N\big( |b_1|_A \|\varphi \|_A+\cdots
+|b_s|_A \|\varphi \|_A \big) \\ & = & N \big( |b_1|_A +\cdots +|b_s|_A
\big) \|\varphi \|_A \\ & \leqslant & NMs\| \varphi \|_A \\ & \leqslant & C\|\varphi \|_A.
\end{array}
 $$
By symmetry, $\|\varphi \|_A \leqslant C\|\varphi \|_B$ and (i) is proved.

To see (ii), given $\Phi \in \Out G$, choose $\varphi \in \Phi$ such that $\| \varphi
\|_A=\|\Phi \|_A$ and then
 $$
\| \Phi \|_B =\min \{ \| \theta \|_B \mid \theta \in \Phi \} \leqslant \| \varphi \|_B
\leqslant C\| \varphi \|_A =C\| \Phi \|_A.
 $$
A symmetric argument completes the proof.
\end{proof}

\begin{prop}\label{canvi}
Let $G$ be a group, and let $A=\{ a_1,\ldots ,a_r\}$ and $B=\{ b_1,\ldots ,b_s\}$ be two
finite generating sets. Then, there exists a constant $C\geqslant 1$ such that, for all
$n\geqslant 1$, the following inequalities hold:
\begin{itemize}
\item[(i)] $\frac1C \cdot \alpha_B(\left\lfloor \frac{n}{C}\right\rfloor )\leqslant
    \alpha_A(n)\leqslant C\cdot \alpha_B(Cn)$,
\item[(ii)] $\frac1C \cdot \beta_B(\left\lfloor \frac{n}{C}\right\rfloor )\leqslant
    \beta_A(n)\leqslant C\cdot \beta_B(Cn)$.
\end{itemize}
\end{prop}

\begin{proof}
For $n=0,1,\ldots ,r-1$, the left and middle terms in both inequalities are zeros and
the result is trivial. For $n\geqslant r$, and using the constant $C$ from the previous
lemma, we have
 $$
\begin{array}{rcl} \alpha_A(n) & = & \max \{ \|\theta^{-1}\|_A \mid \theta \in \aut (G),\,
\|\theta\|_A \leqslant n \} \\ & \leqslant & \max \{ \|\theta^{-1}\|_A \mid \theta \in
\aut (G),\, \|\theta\|_B \leqslant Cn \} \\ & \leqslant & \max \{ C\|\theta^{-1}\|_B
\mid \theta \in \aut (G),\, \|\theta\|_B \leqslant Cn \} \\ & = & C\cdot \max \{ \|
\theta^{-1}\|_B \mid \theta \in \aut (G),\, \|\theta\|_B \leqslant Cn \} \\ & = &
C\cdot \alpha_{B} (Cn).
\end{array}
 $$
By symmetry, $\alpha_B(n)\leqslant C\cdot \alpha_A (Cn)$. Hence, for every $n\geqslant r$,
 $$
\alpha_B(\left\lfloor \frac{n}{C}\right\rfloor )\leqslant C\cdot \alpha_A \left( C\cdot
\left\lfloor \frac{n}{C}\right\rfloor \right)\leqslant C\cdot \alpha_A(n),
 $$
completing the proof of (i).

The exact same argument changing $\alpha$ to $\beta$ proves (ii).
\end{proof}

Straightforward computations show that the following is an equivalence relation on the set
of non-decreasing functions from $\mathbb{N}$ to $\mathbb{N}$: $f\sim g$ if and only if
there exists a constant $C>0$ such that for all $n\geq 0$, $\frac1C \cdot g(\left\lfloor
\frac{n}{C}\right\rfloor )\leqslant f(n)\leqslant C\cdot g(Cn)$. Then,
Proposition~\ref{canvi} is precisely saying that the equivalence classes of the functions
$\alpha_A$ and $\beta_A$ \emph{do not} depend on the set of generators $A$ chosen, that is,
they are invariants of the group $G$. We shall denote them by $\alpha_G$ and $\beta_G$,
respectively.

The relevant information about these (equivalence classes of) functions is their asymptotic
growth. One says that the equivalence class of $f$ grows \emph{at least polynomially with
degree $d$} if there is a constant $L>0$ such that $Ln^d\leqslant f(n)$ for all $n\gg 0$
(i.e. for all $n\geqslant n_0$ and certain $n_0 \geqslant 0$); it is usually said \emph{at
least linearly, quadratically,} or \emph{cubically} when $d=1$, $d=2$, or $d=3$,
respectively. It is also said that $f$ grows \emph{super-polynomially} if it grows at least
polynomially with degree $d$ for every $d>0$. And $f$ grows \emph{exponentially} if there
exists constants $L>0$ and $\lambda>1$ such that $L\lambda^n\leqslant f(n)$ for all $n\gg
0$. One can also define exact growth: $f$ grows \emph{exactly polynomially with degree $d$}
if there are constants $L$ and $M$ such that $Ln^d\leqslant f(n)\leqslant Mn^d$ for all
$n\gg 0$ (which is equivalent to saying $f(n)\sim n^d$). Clearly, all these notions are
well defined not just for functions but for equivalence classes of functions.

Accordingly, we shall use the asymptotic behaviour of the functions $\alpha_G(n)$ and
$\beta_G(n)$ of a given finitely generated group $G$ to define the gap of $G$ for (outer)
automorphism inversion:

\begin{dfn}
\emph{Let $G$ be a finitely generated group and consider the (equivalence classes of)
functions $\alpha_G(n)$ and $\beta_G(n)$. We say that $G$ has \emph{linear} (resp.,
\emph{quadratic, cubic, polynomial of degree $d$, super-polynomial, exponential})
\emph{gap for} [resp., \emph{outer}] \emph{automorphism inversion} if the function
$\alpha_G(n)$ [resp., $\beta_G(n)$] grows linearly (resp., quadratically, cubically,
polynomially of degree $d$, super-polynomially, exponentially).}
\end{dfn}

This notion opens a new direction of research investigating the gap of groups for (outer)
automorphism inversion, by means of analyzing the asymptotic growth of the corresponding
functions. It is easy to see that $\alpha_G(n)$ is equivalent to a constant function if and
only if $|\aut G|<\infty$; similarly, $\beta_G(n)$ is equivalent to a constant function if
and only if $|\Out G|<\infty$. So, in this sense, interesting groups are those with
infinitely many (outer) automorphisms.

Immediately after giving these notions, one can ask many interesting questions which, as
far as we know, are open:

\begin{qs}
Is there a finitely generated group $G$ with super-poly\-no\-mial gap for (outer)
automorphism inversion? And with exponential gap ?
\end{qs}

\begin{qs}
Is there a global upper bound to the gap for (outer) automorphism inversion in the class of
finitely generated groups ? In other words, is it true that given a non-decreasing function
$f\colon \mathbb{N}\to \mathbb{N}$ there exists a finitely generated group $G$ whose gap
for (outer) automorphism inversion grows at least like $f$ ?
\end{qs}

\begin{qs}
Is there a finitely generated group $G$ with $|\Out G|=\infty$ and whose gap for
automorphism inversion is strictly bigger than its gap for outer automorphism inversion ?
\end{qs}

\bigskip

The goal of this paper is to investigate the gap for (outer) automorphism inversion in the
family of finitely generated free groups. For the free group of rank $r$, denoted $F_r$, we
shall write $\alpha_r =\alpha_{F_r}$ and $\beta_r =\beta_{F_r}$.

We can complete this project for the rank two case, which is quite special compared
with higher ranks. On one hand we shall see that, for every free basis $A$ and every
$\Phi \in \Out F_2$, $\| \Phi^{-1}\|_A =\| \Phi \|_A$; hence, $\beta_{2}(n)=n$, while
the same equality in higher rank is far from true. On the other hand, we prove that
$\alpha_2(n)$ is bounded above and below by quadratic functions i.e., $F_2$ has an
exact quadratic gap for automorphism inversion. Collecting
Theorems~\ref{upper},~\ref{lower2} and~\ref{gamma} below, we have

\begin{thm}\hfill
\begin{itemize}
\item[(i)] For $n\geqslant 4$, \ $\alpha_2(n)\leqslant \frac{(n-1)^2}{2}$,
\item[(ii)] for $n\geqslant 10$, \ $\frac{n^2}{4}-6n+42 \leqslant \alpha_2 (n)$,
\item[(iii)] for $n\geqslant 0$, \ $\beta_2(n)= n$.
\end{itemize}
\end{thm}

For higher rank, the problem is much more complicated and our results are less precise. We
show that $\alpha_r(n)$ grows at least polynomially with degree $r$, and $\beta_r(n)$ grows
between polynomially with degree $r-1$, and polynomially with a big enough degree.
Collecting Theorem~\ref{highrank} and Corollary~\ref{highup}, we have

\begin{thm} For every $r\geqslant 3$, there exist constants $K_r, K'_r,
K''_r, M_r > 0$ such that, for every $n\geqslant 0$,
\begin{itemize}
\item[(i)] $K_rn^{r} \leqslant \alpha_r(n)$,
\item[(ii)] $K'_rn^{r-1} \leqslant \beta_r(n)\leqslant K''_rn^{M_r}$.
\end{itemize}
\end{thm}

To our knowledge, nothing else is know about the gap for (outer) automorphism inversion in
free groups of rank bigger than two. In particular, we highlight the following interesting
open questions:

\begin{qs}\label{q}
What is the exact gap for (outer) automorphism inversion in free groups $F_r$, with
$r\geqslant 3$?
\end{qs}

\begin{qs}\label{w}
Is there a polynomial upper bound for the gap for automorphism inversion in free groups
$F_r$, with $r\geqslant 3$?
\end{qs}

\section{Free groups}

\subsection{Notation}

Let $A_r =\{ a_1,\, \ldots ,a_r,\, a_1^{-1},\, \ldots ,\, a_r^{-1} \}$ be an
\emph{alphabet} of $r$ symbols together with their formal inverses (a total of $2r$ symbols
different from each other). All along the paper we assume $r\geqslant 2$ to avoid trivial
cases.

The set of all words on $A_r$, including the empty one denoted 1, together with the
operation of concatenation of words, forms a free monoid denoted $A_r^*$. For any subset
$S\subseteq A_r^*$, the symbol $S^*$ denotes the submonoid generated by $S$, namely the set
of all (arbitrarily long) finite formal products of elements in $S$. For example, $\{
a_1,\, \ldots ,a_r \}^*$ is precisely the set of all \emph{positive} words on the alphabet
$A_r$.

Let $F_r =\langle a_1, \ldots ,a_r\rangle$ be the free group (of rank $r$) on the alphabet
$A_r$, i.e. $A_r^*/\sim$ where $\sim$ is the congruence generated by the elementary
reductions $a_ia_i^{-1}\sim a_i^{-1}a_i\sim 1$. A word of $A_r^*$ is said to be
\emph{(cyclically) reduced} if it contains no (cyclic) factor of the form $a_i^{\epsilon}
a_i^{-\epsilon}$, $\epsilon =\pm 1$. Given a word $w\in A_r^*$, we shall denote by
$\overline{w}$ its reduction, namely the unique reduced word representing the same element
of $F_r$ as $w$. We shall do the standard abuse of notation consisting on using words,
specially reduced ones, to refer to elements of $F_r$.

Note that the length $|w|_A$ of an element $w\in F_r$ is precisely the number of letters in
$\overline{w}$; we shall simplify notation and just denoted it by $|w|$ (there will be no
risk of confusion because, since now on, we shall always work with respect to the
preselected generating set $A$).

\bigskip

Let us consider now automorphisms. Since every $\varphi \in \aut F_r$ is determined by the
images of $a_1,\, \ldots ,\, a_r$, say $a_1\varphi =u_1,\, \ldots ,\, a_r\varphi =u_r$, we
shall adopt the notation $\varphi =\eta_{u_1,\, \ldots ,\, u_r}$, on occasion. When all of
the $u_i$'s are positive words, we say that $\eta_{u_1,\, \ldots ,\,u_r}$ is a
\emph{positive} automorphism (also known in the literature as \emph{invertible
substitution}, see e.g.~\cite{WenWen}). The submonoid of $\aut F_r$ consisting of all
positive automorphisms is denoted by $\aut^+ F_r$. An automorphism $\eta_{u_1,\, \ldots
,\,u_r}$ is said to be \emph{cyclically reduced} when $u_1,\, \ldots ,\,u_r$ are all
cyclically reduced.

As above, we shall also omit the reference to $A$ from the notation for the norm of an
automorphism $\varphi \in \aut F_r$, the norm of an outer automorphism $\Phi \in \Out F_r$,
and also from the gap functions:
 $$
\|\varphi \| = |a_1\varphi| +\cdots +|a_r\varphi|,
 $$
 $$
\|\Phi \| = \min \, \{ \|\varphi \| \mid \varphi \in \Phi \},
 $$
 $$
\alpha_r(n) =\max \, \{ \| \varphi^{-1}\| \mid \varphi \in \aut F_r,\quad \|\varphi\|
\leqslant n\},
 $$
 $$
\beta_r(n) =\max \, \{ \| \Phi^{-1}\| \mid \Phi \in \Out F_r,\quad \|\Phi\| \leqslant n\}.
 $$
Note that there are exactly $r!2^r$ automorphisms with $\| \varphi \|=r$, namely those of
the form $a_1\mapsto a_{1\pi}^{\epsilon_1},\, \ldots ,\, a_r\mapsto a_{r\pi}^{\epsilon_r}$,
where $\pi \in S_r$ is a permutation of $\{ a_1,\, \ldots ,\, a_r\}$ and $\epsilon_i =\pm
1$. These automorphisms are the simplest ones and are called \emph{letter permutation}
automorphisms of $F_r$. They will be useful to reduce the number of cases in our arguments
below.

Observe also that the natural inclusion $\aut F_r \hookrightarrow \aut F_{r+1}$ defined by
fixing the last generator, gives the inequality $\alpha_{r+1}(n+1)\geqslant 1+\alpha_r(n)$.

The following proposition is another reason for omitting the reference to $A$ from the
notation. It presents a stronger form of Proposition~\ref{canvi} when restricting our
attention to free generating sets: given two bases $A$ and $B$ of $F_r$, the functions
$\alpha_A$ and $\alpha_B$ are not only equivalent but exactly equal i.e.,
$\alpha_A(n)=\alpha_B(n)$ for all $n\geqslant 0$. The same is true for the $\beta$
functions.

\begin{prop}\label{cb}
Let $A$ and $B$ be two bases of $F_r$. Then, $\alpha_A(n)=\alpha_B(n)$ and
$\beta_A(n)=\beta_B(n)$, for all $n\geqslant 0$.
\end{prop}

\begin{proof}
Let $\psi \colon F_r \to F_r$ be the automorphism defined by $b_i\psi =a_i$, $i=1,\ldots
,r$. It is clear that, for every $w\in F_r$, $|w|_B=|w\psi|_A$. Now, for every $\varphi \in
\aut F_r$, we have
 $$
\begin{array}{rcl}
\| \varphi \|_B & = & |b_1\varphi |_B +\cdots +|b_r \varphi |_B \\ & = & |a_1\psi^{-1}\varphi
|_B +\cdots +|a_r\psi^{-1}\varphi |_B \\ & = & |a_1\psi^{-1}\varphi \psi|_A +\cdots
+|a_r\psi^{-1}\varphi \psi |_A \\ & = & \| \psi^{-1}\varphi \psi \|_A.
\end{array}
 $$
Furthermore, for every $\Phi \in \Out F_r$, we also have
 $$
\begin{array}{rcl}
\| \Phi \|_B & = & \min \{ \| \varphi \|_B \mid \varphi \in \Phi \} \\ & = & \min \{ \|
\psi^{-1}\varphi \psi \|_A \mid \varphi \in \Phi \} \\ & = & \min \{ \| \nu \|_A \mid \nu \in
\Psi^{-1}\Phi \Psi\} \\ & = & \| \Psi^{-1} \Phi \Psi \|_A,
\end{array}
 $$
where $\Psi =[\psi ]\in \Out F_r$. And from these equalities we deduce that, for every
$n\geqslant 0$,
 $$
\begin{array}{rcl}
\alpha_B(n) & = & \max \{ \|\varphi^{-1}\|_B \mid \varphi \in \aut F_r,\quad \| \varphi
\|_B \leqslant n \} \\ & = & \max \{ \|\psi^{-1}\varphi^{-1}\psi \|_A \mid \varphi \in
\aut F_r,\quad \| \psi^{-1}\varphi \psi \|_A \leqslant n\} \\ & = & \max \{ \|\nu^{-1}\|_A
\mid \nu \in \aut F_r, \quad \| \nu \|_A \leqslant n\} \\ & = & \alpha_A(n).
\end{array}
 $$
A similar argument shows that $\beta_B(n)=\beta_A(n)$.
\end{proof}

\subsection{The $p$-norm of an automorphism}

To prove the main results in the paper, we need to introduce a technical generalization
of the notion of norm for an (outer) automorphism (and its corresponding gap
functions). We shall use standard facts about norms on real (or complex) vectors and
matrices. Recall that the maps $\| {\cdot}\|_p \colon \mathbb{R}^k \to \mathbb{R}$,
$\|(x_1,\ldots ,x_k)\|_p =(|x_1|^p+\cdots +|x_k|^p)^{1/p}$ (for $p\in \mathbb{R}^+$)
and $\| {\cdot}\|_{\infty} \colon \mathbb{R}^k \to \mathbb{R}$, $\|(x_1,\ldots
,x_k)\|_{\infty}=\max \, \{ |x_1|,\ldots ,|x_k|\}$ are \emph{vector norms} i.e., they
satisfy the following axioms: (1) $\|\textbf{x}\|_p \geqslant 0$ with equality if and
only if $\textbf{x}=\textbf{0}$; (2) $\| \mu \textbf{x}\|_p =|\mu |\|\textbf{x}\|_p$;
and (3) $\|\textbf{x}+\textbf{y}\|_p \leqslant \|\textbf{x}\|_p +\|\textbf{y}\|_p$.

Let us extend these notions to the non-abelian context, via the length function. For $p\in
\overline{\mathbb{R}}^+ =\mathbb{R}^+\cup \{ \infty\}$ and $\textbf{w}=(w_1,\ldots, w_k)\in
F_r^k$, we define
 $$
\| \textbf{w}\|_p =\| (w_1, \ldots ,w_k)\|_p = (|w_1|^p +\cdots +|w_k|^p )^{1/p}
 $$
for $p\in \mathbb{R}^+$, and
 $$
\| \textbf{w}\|_{\infty} =\| (w_1, \ldots ,w_k)\|_{\infty} = \max \, \{ |w_1|,\ldots
,|w_k|\}
 $$
for $p=\infty$. Note that the notation is coherent with the fact $\| \textbf{w}\|_{\infty}
=\lim_{p\to \infty} \| \textbf{w} \|_p$.

Observe that this map $F_r^k \to \mathbb{R}$ can be expressed in terms of the corresponding
vector norm, $\| (w_1, \ldots ,w_k)\|_p =\| (|w_1|, \ldots ,|w_k| )\|_p$. Hence, it
satisfies the following properties:
\begin{itemize}
\item[1)] (positivity) $\| \textbf{w}\|_p \geqslant 0$, with equality if and only if
    $\textbf{w}=(1,\ldots ,1)$;
\item[2)] (powers) $\| (w_1^n,\ldots ,w_k^n)\|_p \leqslant |n|\|(w_1,\ldots ,w_k)\|_p$;
\item[3)] (triangular inequality) $\| (v_1w_1, \ldots ,v_kw_k)\|_p \leqslant \|
    (v_1,\ldots ,v_k)\|_p +\| (w_1,\ldots ,w_k) \|_p$.
\end{itemize}
By analogy, we shall refer to these three properties by naming $\| {\cdot}\|_p$ as the
\emph{$p$-norm} in $F_r^k$.

Let us move now to morphisms. Thinking of endomorphisms of $F_r$ (and, in particular,
automorphisms) as $r$-tuples of elements, $\varphi \leftrightarrow (a_1\varphi, \ldots
,a_r\varphi)$, we define the \emph{$p$-norm} of an endomorphism $\varphi \in \End F_r$,
$p\in \overline{\mathbb{R}}^+$, as
 $$
\| \varphi \|_p =\|(a_1\varphi, \ldots ,a_r\varphi)\|_p.
 $$
Given $\Phi \in \Out F_r$, define also
 $$
\|\Phi \|_p = \min \, \{ \|\varphi \|_p \mid \varphi \in \Phi \}.
 $$
Of course, $\|\varphi \|_1$ and $\|\Phi \|_1$ equal, respectively, the values $\|\varphi\|$
and $\|\Phi \|$ defined in the previous section.

Further, we define the corresponding gap functions $\alpha^p_r$ and $\beta^p_r$ in the
natural way:
\begin{align*}
\alpha_r^p(n) & = \ \max \, \{ \| \varphi^{-1}\|_p \mid \varphi \in \aut F_r,\quad
\|\varphi\|_p \leqslant n\},\\[+.2cm] \beta_r^p(n) & =  \ \max \, \{ \| \Phi^{-1}\|_p \mid
\Phi \in \Out F_r,\quad \|\Phi \|_p \leqslant n\}.
\end{align*}
Clearly, these are non-decreasing functions from $\mathbb{N}$ to $\mathbb{R}$. Again,
$\alpha_r$ and $\beta_r$ from the previous section are just $\alpha_r^1$ and $\beta_r^1$,
respectively. Furthermore, the following proposition states that the functions $\alpha_r^p$
belong to the same equivalence class for all different values of $p\in
\overline{\mathbb{R}}^+$; the same happens for the functions $\beta_r^p$ (note that the
equivalence relation defined above for functions from $\mathbb{N}$ to $\mathbb{N}$ can
naturally be extended to functions from $\mathbb{N}$ to $\mathbb{R}$). For this reason, we
shall restrict our attention to the case $p=1$ (with occasional references to the
$\infty$-norm for some technical arguments).

\begin{prop}\label{constants}
For all $p,q\in \overline{\mathbb{R}}^+$ there exists a natural number $C=C_{p,q,r}>0$ such
that
 $$
\frac{1}{C}\|\varphi\|_q \leqslant \|\varphi \|_p \leqslant C\|\varphi \|_q \quad \text{ and
}\quad \frac{1}{C}\| \Phi \|_q \leqslant \| \Phi \|_p \leqslant C\| \Phi \|_q
 $$
hold for all $\varphi \in \End F_r$ and $\Phi \in \Out F_r$. Furthermore, for all
$n\geqslant 0$,
 $$
\frac{1}{C}\alpha_r^p\left( \left\lfloor\frac{n}{C}\right\rfloor \right) \leqslant
\alpha_r^q(n) \leqslant C\alpha_r^p (Cn),
 $$
 $$
\frac{1}{C}\beta_r^p\left(\left\lfloor \frac{n}{C}\right\rfloor \right)\leqslant \beta_r^q(n)
\leqslant C\beta_r^p (Cn).
 $$
\end{prop}

\begin{proof} It is well-known (see \cite[Corollary~5.4.5]{HJ}) that the exact similar fact holds
for the corresponding vector norms: there exists a positive constant, and so a natural
number $C=C_{p,q,r}$, such that
 $$
\frac{1}{C}\|\textbf{x}\|_q \leqslant \|\textbf{x}\|_p \leqslant C\|\textbf{x}\|_q
 $$
for every $\textbf{x}\in \mathbb{R}^r$. Now $ \frac{1}{C}\|\varphi\|_q \leqslant \|\varphi
\|_p \leqslant C\|\varphi \|_q$ follows immediately from the equality
 $$
\| \varphi \|_p =\|(a_1\varphi, \ldots ,a_r\varphi)\|_p =\|(|a_1\varphi|, \ldots
,|a_r\varphi |)\|_p.
 $$
On the other hand, since $\| \Phi \|_q = \| \theta \|_q$ for some $\theta \in \Phi$, we get
 $$
\| \Phi \|_p = \min \,\{ \|\varphi \|_p \mid \varphi \in \Phi \} \leqslant \| \theta \|_p
\leqslant C\| \theta \|_q = C\| \Phi \|_q
 $$
and $\frac{1}{C}\| \Phi \|_q \leqslant \| \Phi \|_p \leqslant C\| \Phi \|_q$ follows by
symmetry.

For the second part of the statement, we have
\begin{align*}
\alpha_r^q (n) & =  \ \max \, \{ \|\varphi^{-1} \|_q \mid \varphi \in \aut F_r,\quad \|
\varphi\|_q \leqslant n \} \\ {} & \leqslant \ \max \, \{ \|\varphi^{-1} \|_q \mid \varphi
\in \aut F_r,\quad \| \varphi\|_p \leqslant Cn \} \\ {} & \leqslant \ C \max \, \{
\|\varphi^{-1} \|_p \mid \varphi \in \aut F_r,\quad \| \varphi\|_p \leqslant Cn \} \\
 {} & = \ C\alpha_r^p (Cn)
\end{align*}
for all $n$. Symmetrically, $\alpha_r^p (n) \leqslant C\alpha_r^q (Cn)$. Now, for every
natural number $n$, write $C\lfloor\frac{n}{C}\rfloor \leqslant n$ and we have $\alpha_r^p
\left( \lfloor\frac{n}{C}\rfloor\right) \leqslant C\alpha_r^q \left( C\lfloor\frac{n}{C}
\rfloor\right)\leqslant C\alpha_r^q (n)$ and so, $\frac{1}{C}\alpha_r^p\left(
\lfloor\frac{n}{C}\rfloor \right)\leqslant \alpha_r^q(n)$.

The same argument gives the corresponding inequalities for the $\beta$ functions.
\end{proof}

The following lemmas state some basic properties of norms of automorphisms and outer
automorphisms of free groups, that will be useful later.

\begin{lem}\label{compo}
Let $\varphi,\, \theta,\, \psi_1,\, \psi_2 \in \aut F_r$ with $\psi_1$ and $\psi_2$ letter
permuting, and let $w \in F_r \setminus \{ 1\}$. Then:
\begin{itemize}
\item[(i)] $\frac{\| \varphi \|_1}{r}\leqslant \| \varphi \|_{\infty} <\| \varphi \|_1$,
\item[(ii)] $\|\psi_1 \varphi \psi_2 \|_p =\| \varphi \|_p$ for all $p\in
    \overline{\mathbb{R}}^+$,
\item[(iii)] $||\varphi\theta||_1 \leqslant ||\varphi||_1 \cdot ||\theta||_{\infty}
    <||\varphi||_1 \cdot ||\theta||_1$,
\item[(iv)] $||\lambda_w\varphi||_1 \leqslant (2r|w|+r-2)||\varphi||_{\infty} <
    (2r|w|+r-2)||\varphi||_1$.
\end{itemize}
\end{lem}

\begin{proof} (i) and (ii) are clear from the definitions.

(iii) For every $a\in A_r$, we have $|a\varphi\theta| \leqslant |a\varphi|\cdot
||\theta||_{\infty}$ and so
 $$
||\varphi\theta||_1 = \sum_{i=1}^r |a_i\varphi\theta| \leqslant \sum_{i=1}^r |a_i\varphi|
\cdot ||\theta||_{\infty} = ||\varphi||_1 \cdot ||\theta||_{\infty} <||\varphi||_1 \cdot
||\theta||_1.
 $$

(iv) Since $w \neq 1$, exactly one of the words $w^{-1} a_i w$ is non reduced, and so

\begin{align*}
||\lambda_w\varphi||_1 & =  \  \sum_{i=1}^r |(\overline{w^{-1} a_iw})\varphi| \leqslant
(r-1)(2|w|+1)||\varphi||_{\infty} + (2|w|-1)||\varphi||_{\infty} \\ {} & = \  (2r|w|+r-2)
||\varphi||_{\infty} < (2r|w|+r-2)||\varphi||_1. \qedhere
\end{align*}
\end{proof}

\begin{lem}\label{outcompo}
Let $\Phi,\, \Theta \in \Out F_r$ and let $\psi_1,\, \psi_2 \in \aut F_r$ be letter
permuting. Then:
\begin{itemize}
\item[(i)] $\| [\psi_1] \Phi [\psi_2] \|_1 =\| \Phi \|_1$,
\item[(ii)] $\| \Phi\Theta \|_1 \leqslant \| \Phi \|_1 \| \Theta \|_1$.
\end{itemize}
\end{lem}

\begin{proof}
We have $[\psi_1] \Phi [\psi_2] = \psi_1\Lambda_r \Phi \psi_2 \Lambda_r = \psi_1\Lambda_r
\Phi \Lambda_r \psi_2 = \psi_1 \Phi \psi_2$. Now Lemma~\ref{compo}(ii) yields
 $$
\| [\psi_1] \Phi [\psi_2] \|_1 = \min \, \{ \| \psi_1\varphi\psi_2 \|_1 \mid \varphi \in
\Phi \} = \min \, \{ \| \varphi \|_1 \mid \varphi \in \Phi \} =\| \Phi \|_1
 $$
and so (i) holds.

For (ii), we use Lemma~\ref{compo}(iii) to get
\begin{align*}
\| \Phi\Theta \|_1&= \ \min \, \{ \| \psi \|_1 \mid \psi \in \Phi\Theta \} =\min \, \{ \|
\varphi\theta \|_1 \mid \varphi \in \Phi,\; \theta \in \Theta \} \\ & \leqslant \ \min \, \{
\| \varphi\|_1 \|\theta \|_1 \mid \varphi \in \Phi,\; \theta \in \Theta \} \\
&=(\min \, \{ \|
\varphi\|_1 \mid \varphi \in \Phi \})(\min \, \{ \| \theta \|_1 \mid \theta \in \Theta \})
= \ \| \Phi \|_1 \| \Theta \|_1 \,. \qedhere
\end{align*}
\end{proof}



\begin{lem}
\label{ocr} Let $\varphi \in \aut F_r$ be cyclically reduced. Then  $\| [\varphi] \|_1 = \|
\varphi \|_1$.
\end{lem}

\subsection{Abelianization}

Abelianization will be a valuable tool to derive lower bounds for $\| \varphi \|_1$ and $\|
\Phi\|_1$.

The 1-norm for vectors $\|(x_1,\ldots ,x_r)\|_1=|x_1|+\cdots +|x_r|$ gives rise to the
1-norm for matrices, namely
 $$
\|M\|_1 =\sum_{i,j} |m_{i,j}|,
 $$
where $M=(m_{i,j})\in \GL_r(\mathbb{Z})$. It is straightforward to verify that, for all
$\textbf{x},\, \textbf{y}\in \mathbb{Z}^r$ and $M,N\in \GL_r(\mathbb{Z})$, we have the
inequalities $\| \textbf{x}+\textbf{y}\|_1 \leqslant \| \textbf{x}\|_1 +\| \textbf{y}\|_1$,
$\| \textbf{x}M\|_1 \leqslant \| \textbf{x}\|_1 \cdot \| M\|_1$, $\| M+N\|_1 \leqslant \|
M\|_1 +\| N\|_1$, and $\| MN\|_1 \leqslant \| M\|_1 \| N\|_1$.

Let us denote the abelianization map by $({\cdot})^{\rm ab}\colon F_r \twoheadrightarrow
\mathbb{Z}^r$, $w\mapsto w^{\rm ab}=([w]_{a_1},\, \ldots ,\, [w]_{a_r})$. Here, $[w]_{a_i}$
is the total exponent of $a_i$ in $w$, i.e. the total number of times the letter $a_i$
occurs in $\overline{w}$, taking into account the exponents' signs (for example,
$[a_1a_2a_1^{-2}]_{a_1}=-1$ and $[a_1a_1^{-1}a_2]_{a_1}=[a_2]_{a_1}=0$).

Every automorphism $\varphi \in \aut F_r$ abelianizes to an automorphism $\varphi^{\rm
ab}$ of $\mathbb{Z}^r$ which we shall represent by its $r\times r$ (invertible) matrix
over $\mathbb{Z}$. We want automorphisms to act on the right, and so we write matrices
by rows i.e., with the $i$-th row describing the image of the $i$-th generator:
 $$
\varphi^{\rm ab} = \begin{pmatrix} [a_1 \varphi ]_{a_1} & \cdots & [a_1 \varphi ]_{a_r} \\
\cdots & \cdots & \cdots \\ \,[a_r \varphi ]_{a_1} & \cdots & [a_r \varphi ]_{a_r}
\end{pmatrix} \in \GL_r(\mathbb{Z}).
 $$
This way, for every $w\in F_r$, $(w\varphi )^{\rm ab} =w^{\rm ab}\varphi^{\rm ab}$.
Furthermore, $(\varphi \theta)^{\rm ab}=\varphi^{\rm ab} \theta^{\rm ab}$, and
$(\varphi^{-1})^{\rm ab}=(\varphi^{\rm ab})^{-1}$.

Observe that, for every $w\in F_r$, $|w| \geqslant \| w^{\rm ab}\|_1 =|[w]_{a_1}|+\cdots
+|[w]_{a_r}|$ with equality if and only if no letter occurs in $\overline{w}$ with the two
opposite signs. This can be expressed in the following useful way:

\begin{lem}\label{ab}
For every $\varphi \in \aut F_r$, $\| \varphi\|_1 \geqslant \| [\varphi] \|_1 \geqslant \|
\varphi^{\rm ab}\|_1$, with equalities if and only if, for every $i=1,\ldots ,r$, no letter
occurs in $\overline{a_i\varphi}$ with the two opposite signs. In particular, $\| \varphi
\|_1 =\| \varphi^{\rm ab} \|_1$ for positive automorphisms.
\end{lem}

\begin{proof}
Clearly, $\| \varphi\|_1 \geqslant \| [\varphi] \|_1$. We may write $\| [\varphi] \|_1 = \|
\varphi\lambda_w \|_1$ for some $w \in F_r$. Then
\begin{align*}
 \| \varphi\|_1&\geqslant\| [\varphi] \|_1 =\| \varphi\lambda_w \|_1
=\sum_{i=1}^r |a_i\varphi\lambda_w| \geqslant \sum_{i=1}^r \| (a_i\varphi )^{\rm ab}\|_1 \\
&=\sum_{i=1}^r \| a_i^{\rm ab} \varphi^{\rm ab}\|_1 =\sum_{i=1}^r \sum_{j=1}^r |[a_i
\varphi ]_{a_j}| =\| \varphi^{\rm ab}\|_1,
\end{align*}
where $a_i^{\rm ab}$ is the $i$-th canonical vector and so, $a_i^{\rm ab} \varphi^{\rm ab}$
is the $i$-th row in $\varphi^{\rm ab}$. It is immediate that the inequality $\|
\varphi\|_1 \geqslant \| \varphi^{\rm ab}\|_1$ becomes an equality if and only if, for
every $i=1,\ldots ,r$, no letter occurs in $\overline{a_i\varphi}$ with the two opposite
signs. This is the case when $\varphi \in \aut^+ F_r$.
\end{proof}

\section{The rank two case}

In this section we shall deal with the rank 2 case. For the duration of this section, we
simplify our notation to $A=A_2 =\{ a,b,a^{-1}, b^{-1}\}$.

We start by proving that inversion preserves the norm in the case of positive
automorphisms. It is known that positive automorphisms of $F_2$ are generated as a
monoid by $\Delta =\{ \eta_{b,a},\, \eta_{a,ab},\, \eta_{a,ba} \}$, that is, they all
can be obtained as a composition of these elementary ones i.e., $\aut^+ F_2=\Delta^*$
(see~\cite{WenWen}).

\begin{lem}\label{inv+}
Let $\varphi \in \aut^+ F_2$ and write $\varphi^{-1}=\eta_{u,v}$. Then either $u\in \{
a,b^{-1}\}^*$ and $v\in \{ a^{-1},b\}^*$, or $u\in \{ a^{-1},b\}^*$ and $v\in \{
a,b^{-1}\}^*$. In particular, $\varphi^{-1}$ is cyclically reduced.
\end{lem}

\begin{proof}
The result is clear for the three elementary positive automorphisms, $\eta_{b,a}^{-1}
=\eta_{b,a}$, $\eta_{a,ab}^{-1} =\eta_{a,a^{-1}b}$, $\eta_{a,ba}^{-1} =\eta_{a,ba^{-1}}$.
Since all positive automorphisms are compositions of elements from $\Delta$, it is
sufficient to show that, given a positive automorphism $\varphi$ and $\theta \in \Delta$,
the lemma holds for $\varphi \theta$ whenever it holds for $\varphi$. To see this, write
$\varphi^{-1}=\eta_{u,v}$ and assume $u$ and $v$ are as in the statement. Then we get
\begin{align*}
(\varphi \eta_{b,a})^{-1} & =\eta_{b,a} \eta_{u,v}=\eta_{v,u}, \\
(\varphi \eta_{a,ab})^{-1} & =\eta_{a,a^{-1}b} \eta_{u,v}=\eta_{u,u^{-1}v}, \\
(\varphi \eta_{a,ba})^{-1} &  =\eta_{a,ba^{-1}} \eta_{u,v}=\eta_{u,vu^{-1}},
\end{align*}
completing the proof.
\end{proof}

\begin{prop}\label{pos}
Let $\varphi \in \aut^+ F_2$. Then $\| \varphi^{-1}\|_1=\| \varphi \|_1$.
\end{prop}

\begin{proof}
 Abelianizing, we have
 $$
\varphi^{\rm ab}= \left( \begin{array}{cc} [a\varphi ]_a & [a\varphi ]_b \\ \,[b\varphi ]_a
& [b\varphi ]_b \end{array} \right) \quad \text{ and } \quad (\varphi^{-1})^{\rm ab}=\pm
\left( \begin{array}{rr} [b\varphi ]_b & -[a\varphi ]_b \\ -[b\varphi ]_a & [a\varphi ]_a
\end{array}\right);
 $$
hence, $\| (\varphi^{-1})^{\rm ab}\|_1 =\| \varphi^{\rm ab}\|_1$. Also, $\| \varphi^{\rm
ab}\|_1 =\| \varphi \|_1$ since $\varphi$ is positive (see Lemma~\ref{ab}). Now, write
$\varphi^{-1}=\eta_{u,v}$. By Lemma~\ref{inv+} no letter occurs with both signs in neither
$u$ nor $v$ so, again by Lemma~\ref{ab}, $\| (\varphi^{-1})^{\rm ab}\|_1 =\| \varphi^{-1}
\|_1$, concluding the proof.
\end{proof}

From positive automorphisms we can gain control of all cyclically reduced ones.

\begin{lem}\label{+}
For every cyclically reduced $\varphi \in \aut F_2$, there exist two letter permuting
automorphisms $\psi_1,\, \psi_2 \in \aut F_2$ and $\theta \in \aut^+ F_2$ such that
$\varphi =\psi_1 \theta \psi_2$.
\end{lem}

\begin{proof}
 Write $\varphi =\eta_{u,v}$. Since both $u$ and $v$ are cyclically reduced, the main
result in~\cite{CMZ} tells us that at most two letters of $A$ occur in $u$, and at most two
of them (not necessarily the same ones) occur in $v$. Without loss of generality, we may
assume that two different letters occur in either $u$ or $v$, say in $u$. Inverting all
possibly negative letters in $u$, we can write $\eta_{u,v}= \eta_{u',v'}\eta_{a^{\epsilon},
b^{\delta}}$ with $\epsilon,\, \delta =\pm 1$, $u'\in \{ a,b\}^*$ and $|u'|=|u|$ and
$|v'|=|v|$.

If $v' \in \{ a,\,b \}^*$ i.e., it is a positive word, then $\eta_{u',v'}\in \aut^+
F_2$ and we are done. If $v'\in \{ a^{-1},\, b^{-1}\}^*$, take $\eta_{u,v}=\eta_{a,
b^{-1}}\eta_{u', v'^{-1}} \eta_{a^{\epsilon},b^{\delta}}$ and we are also done. The
remaining cases to consider are $v' \in \{ a^{-1},\, b\}^*$ or $v'\in \{ a,\,
b^{-1}\}^*$ with exactly two letters occurring in $v'$; they will lead us to
contradiction. Indeed, abelianizing, we get $u'^{\rm ab}=([u]_a,\, [u]_b)=(p,q)$ with
$p,q>0$, and $v'^{\rm ab}=([v]_a,\, [v]_b)=(r,s)$ with $rs<0$. This contradicts
$ps-qr=\pm 1$ coming from the fact that $\eta_{u',v'}$ is an automorphism of $F_2$.
\end{proof}

And from those, we can reach the general case:

\begin{lem}\label{+ncr}
For every $\varphi \in \aut F_2$, there exist two letter permuting automorphisms $\psi_1,\,
\psi_2 \in \aut F_2$, $\theta \in \aut^+ F_2$, and an element $g\in F_2$ such that $\varphi
=\psi_1 \theta \psi_2 \lambda_g$ and $\| \theta \|_1 +2|g|\leqslant \| \varphi \|_1$.
\end{lem}

\begin{proof}
Note that, by Lemmas~\ref{compo}(ii) and \ref{+}, it suffices to show that there exists a
cyclically reduced $\varphi'\in \aut F_2$ and $g\in F_2$, such that $\varphi =\varphi'
\lambda_g$ and $\| \varphi'\|_1 +2|g| \leqslant \| \varphi \|_1$. Let us prove this claim
by induction on $\| \varphi \|_1$.

If $\| \varphi \|_1 =2$ the claim is trivial since $\varphi$ is already cyclically reduced.
So, suppose $\varphi =\eta_{u,v} \in \aut F_2$ is given with $\| \eta_{u,v} \|_1 >2$, and
let us assume the claim holds for all automorphisms of smaller 1-norm. Again, if $u$ and
$v$ are cyclically reduced the claim is trivial so, by symmetry, we can assume that $u$ is
not cyclically reduced, say $\overline{u}=c^{-1}u'c$ for some $c\in A$ and $u'\in F_2$. If
$\overline{v}$ neither begins with $c^{-1}$ nor ends with $c$ then it could be easily seen
that $c$ would not be contained in $\langle u,v\rangle$ contradicting the fact that $\{ u,
v\}$ generates $F_2$. Hence, $v\in c^{-1}A^* \cup A^*c$, and so $|\overline{cvc^{-1}}|
\leqslant |v|$. Now, factoring $\eta_{u,v}$ as $\eta_{u,v}=\eta_{u', \overline{cvc^{-1}}}
\lambda_{c}$, we have
 $$
\| \eta_{u',\overline{cvc^{-1}}}\|_1 =|u'|+|\overline{cvc^{-1}}| \leqslant |u|-2+|v| =\|
\eta_{u,v} \|_1 -2,
 $$
and we can apply the induction hypothesis to get a factorization
$\eta_{u',\overline{cvc^{-1}}} =\varphi' \lambda_h$ with $\varphi'$ cyclically reduced and
$\| \varphi'\|_1 +2|h| \leqslant \| \eta_{u',\overline{cvc^{-1}}} \|_1$. Thus, we have
$\eta_{u,v}=\eta_{u',\overline{cvc^{-1}}}\lambda_{c} =\varphi' \lambda_h\lambda_{c}
=\varphi' \lambda_{hc}$ with
 $$
\| \varphi' \|_1 +2|hc| \leqslant \| \varphi' \|_1 +2|h| +2 \leqslant \| \eta_{u',
\overline{cvc^{-1}}} \|_1 +2\leqslant \| \eta_{u,v} \|_1 =\| \varphi \|_1.
 $$
This completes the proof of the claim and so, of the lemma.
\end{proof}

\begin{thm}\label{upper}
For every $n\geqslant 4$, we have $\alpha_2(n)\leqslant \frac{(n-1)^2}{2}$.
\end{thm}

\begin{proof}
 Let $\varphi \in \aut F_2$ with $\| \varphi \|_1 \leqslant n$, and let us prove that
$\| \varphi^{-1}\|_1 \leqslant \frac{(n-1)^2}{2}$. Consider the decomposition given in
Lemma~\ref{+ncr}, $\varphi =\psi_1 \theta \psi_2 \lambda_g$ for some letter permuting
$\psi_1,\, \psi_2 \in \aut F_2$, some $\theta \in \aut^+ F_2$, and some $g\in F_2$ such
that $\| \theta \|_1 +2|g|\leqslant \| \varphi \|_1$.

If $g=1$ then
 $$
\| \varphi^{-1}\|_1 =\| \psi_2^{-1}\theta^{-1}\psi_1^{-1}\|_1 =\| \theta^{-1}\|_1 =\| \theta
\|_1 =\| \varphi \|_1 \leqslant n\leqslant \frac{(n-1)^2}{2},
 $$
by Lemma~\ref{compo}(ii) and Proposition~\ref{pos} (and using in the last step that
$n\geqslant 4$).

So, let us assume $g\neq 1$ in which case we have $\varphi^{-1}=\lambda_{g^{-1}}
\psi_2^{-1} \theta^{-1} \psi_1^{-1}$. By Lemma~\ref{compo} and Proposition~\ref{pos},
\begin{align*}
\| \varphi^{-1}\|_1 & \leqslant 4|g| \cdot \| \psi_2^{-1} \theta^{-1} \psi_1^{-1} \|_{\infty}
=4|g|\cdot \| \theta^{-1} \|_{\infty} \leqslant 4|g|(\| \theta^{-1} \|_1 -1) \\
 & =4|g|(\| \theta \|_1 -1).
 \end{align*}
Since we also have $\| \theta \|_1 +2|g|\leqslant \| \varphi \|_1 \leqslant n$, we deduce
$|g|\leqslant \frac{n-\| \theta\|_1}{2}$ and so,
 $$
\| \varphi^{-1} \|_1 \leqslant 2(n-\| \theta \|_1)(\| \theta \|_1 -1).
 $$
Finally, since the parabola $f(x)=2(n-x)(x-1)$ has its absolute maximum in the point
$x=\frac{n+1}{2}$, we conclude
 $$
\| \varphi^{-1} \|_1 \leqslant 2(n-\| \theta \|_1)(\| \theta \|_1 -1)\leqslant 2\Big( n-
\frac{n+1}{2}\Big) \Big( \frac{n+1}{2}-1 \Big)=\frac{(n-1)^2}{2}. \qedhere
 $$
\end{proof}

In order to establish lower bounds for $\alpha_2 (n)$, we need to construct explicit
automorphisms of $F_2$ having inverses with 1-norm much bigger than that of themselves.

\begin{thm}\label{lower2}
For $n\geqslant 10$, we have $\alpha_2 (n) \geqslant \frac{n^2}{4}-6n+42$.
\end{thm}

\begin{proof}
For $k\geqslant 0$ consider the automorphisms
 $$
\psi_k =\eta_{ab^{2k},\, ab^{2k+1}}\lambda_{a^{-k}b} =\eta_{b^{-1}a^{k+1}b^{2k}a^{-k}b,\,
b^{-1}a^{k+1}b^{2k+1}a^{-k}b}.
 $$
We have $\| \psi_k \|_1 =8k+7$. For the inverse, we have
 $$
\psi_k^{-1} =\lambda_{b^{-1}a^k} \eta^{-1}_{ab^{2k},\, ab^{2k+1}} =\lambda_{b^{-1}a^k}
\eta_{a(b^{-1}a)^{2k},\, a^{-1}b} =\eta_{u,\, v},
 $$
where $u$ and $v$ are the two words
 $$
u=((a^{-1}b)^{2k}a^{-1})^k a^{-1}ba(b^{-1}a)^{2k} b^{-1}a(a(b^{-1}a)^{2k})^k,
 $$
 $$
v=((a^{-1}b)^{2k}a^{-1})^k a^{-1}b(a(b^{-1}a)^{2k})^k.
 $$
Hence, $\| \psi_k^{-1} \|_1 = 4(4k+1)k+4k+7=16k^2+8k+7$.

Writing $n=\| \psi_k \|_1 =8k+7$, we have $k=\frac{n-7}{8}$ and then
 $$
\| \psi_k^{-1} \|_1 =16\frac{(n-7)^2}{64}+n-7+7=\frac{n^2-10n+49}{4}.
 $$
Thus, for $n\equiv 7 \mod 8$, we have $\alpha_2(n)\geqslant \frac{n^2-10n+49}{4}$.

Finally, for  every $n\geqslant 7$, let $n'$ be the unique integer congruent with 7 modulo
8 in the set $\{ n-7, \ldots ,n-1, n\}$. We have
 \begin{align*}
\alpha_2(n)&\geqslant \alpha_2(n')\geqslant \frac{n'^2-10n'+49}{4}\geqslant \frac{(n-7)^2-
10(n-7)+49}{4} \\&=\frac{n^2}{4}-6n+42,
\end{align*}
where the last inequality uses $n\geqslant 10$ since the parabola
$f(x)=\frac{x^2-10x+49}{4}$ has its minimum at $x=5$.
\end{proof}

The outer automorphism case turns out to be simpler:

\begin{thm}\label{gamma}
For every $\Phi \in \Out F_2$, $\| \Phi^{-1}\|_1 = \| \Phi \|_1$. Consequently,
$\beta_2(n)=n$.
\end{thm}

\begin{proof}
Take $\varphi \in \Phi$. By Lemma~\ref{+ncr}, $\varphi =\psi_1 \theta \psi_2 \lambda_g$
for some letter permuting automorphisms $\psi_1,\, \psi_2 \in \aut F_2$, some $\theta
\in \aut^+ F_2$ and some element $g\in F_2$. Then, Lemmas~\ref{outcompo}(i) and
\ref{ocr} yield
 $$
\| \Phi \|_1 = \| [\varphi] \|_1 = \| [\psi_1 \theta \psi_2 \lambda_g] \|_1 =  \| [\psi_1
\theta \psi_2] \|_1 = \| [\theta] \|_1 = \| \theta \|_1.
 $$
Also, by Lemma~\ref{outcompo}(i), we get
  $$
\| \Phi\inv \|_1 = \| [\varphi\inv] \|_1 = \| [\lambda_{g\inv}\psi_2\inv \theta\inv \psi_1
\inv] \|_1 =  \| [\psi_2\inv \theta\inv \psi_1\inv] \|_1 = \| [\theta\inv] \|_1.
 $$
Since $\theta\inv$ is cyclically reduced by Lemma~\ref{inv+}, we may use Lemma~\ref{ocr} to
get $\| \Phi\inv \|_1 = \| [\theta\inv] \|_1 = \| \theta\inv \|_1$. Since $\| \theta \|_1 =
\| \theta\inv \|_1$ by Proposition~\ref{pos}, we get  $\| \Phi^{-1}\|_1 = \| \Phi \|_1$.
Therefore $\beta_2(n)=n$.
\end{proof}

\section{Higher rank}
\label{higherrank}

In this section, we consider arbitrary rank $r\geqslant 2$, compute polynomial lower bounds
for both  $\alpha_r(n)$ and $\beta_r(n)$, and show that $\beta_r(n)$ admits a polynomial
upper bound.

The polynomial lower bounds for $\alpha_r(n)$ and $\beta_r(n)$ have degrees $r$ and
$r-1$, respectively. In particular, this separates the asymptotic behaviour of the rank
two case from all other ranks, with respect to both complexity functions. That is,
$\xi_2(n)$  grows more slowly than $\xi_r(n)$ for all $r\geqslant 3$ and $\xi \in \{
\alpha, \beta \}$, which agrees with the intuitive fact that $\aut F_r$ is a much
easier group to deal with for $r=2$ than for higher rank.

Finally, the polynomial upper bound for $\beta_r(n)$ is established with the help of
the theory of Outer space.

We assume the rank $r$ fixed throughout the whole section.

\subsection{Lower bounds}

Our lower bound for $\beta_r(n)$ is obtained by abelianization of positive automorphisms.
The extra unit in the degree of the lower bounds from $\beta_r(n)$ to $\alpha_r(n)$ will be
achieved by additionally composing the positive automorphisms with a suitable conjugation
that increases in size when inverting. We thank Warren Dicks for suggesting us to use the
following automorphisms; this significantly simplified our previous proof of the lower
bounds for $\alpha_r(n)$ and $\beta_r(n)$.

We start by defining, for every $p \in \mathbb{Z}$, a matrix $M^{(p)} = (m^{(p)}_{i,j})\in
\GL_r(\mathbb{Z}) = \aut \mathbb{Z}^r$ given by
  $$
m^{(p)}_{i,j} = \begin{cases} 1, &\text{ if }i = j; \\ p, &\text{ if }  j = i+1; \\ 0,
&\text{ otherwise}. \end{cases}
  $$

Note that $\det M^{(p)} = 1$ and so $M^{(p)}$ is indeed invertible.

\begin{lem}
\label{coin} For all $r\geqslant 2$ and $p \in \mathbb{Z}$, let $N^{(p)} = (n^{(p)}_{i,j})
\in \GL_r(\mathbb{Z})$ be defined by
 $$
n^{(p)}_{i,j} = \begin{cases} 1, &\text{ if }i = j; \\ (-p)^{j-i}, &\text{ if } i< j; \\ 0,
&\text{ otherwise}. \end{cases}
 $$
Then $N^{(p)} = (M^{(p)})\inv$.
\end{lem}

\begin{proof}
It suffices to show that $M^{(p)}N^{(p)}$ is the identity matrix. Indeed, the $(i,j)$-th
entry of the product matrix is $\sum_{k=1}^r m^{(p)}_{i,k}n^{(p)}_{k,j}=
\sum_{k=i}^{\min\{i+1,j\}} m^{(p)}_{i,k}n^{(p)}_{k,j}$ which is 0 if $j<i$ and 1 if $j=i$.
If $j>i$, we get $m^{(p)}_{i,i}n^{(p)}_{i,j}+m^{(p)}_{i,i+1}n^{(p)}_{i+1,j}=(-p)^{j-i} +
p(-p)^{j-i-1} = 0$ and the lemma is proved.
\end{proof}

We immediately obtain:

\begin{lem}
\label{invcoin} For all $r\geqslant 2$ and $p \in \mathbb{Z}$, we have $\| M^{(p)} \|_1 =r
+(r-1)p$ and $\| (M^{(p)})\inv \|_1 \geq p^{r-1}$. $\Box$
\end{lem}

For every integer $p\geqslant 2$, define $\varphi_p \in \aut^+ F_r$ by
 $$
a_i\varphi_p=\begin{cases} a_ia_{i+1}^p, &\text{ if }  \ 1 \leq i<r ; \\ a_r, &\text{ if } \
i=r. \end{cases}
  $$
Note that $\varphi_p$ is clearly onto and therefore an automorphism since free groups of
finite rank are hopfian \cite{LS}.

\begin{lem}\label{tecnic}
For all $r\geqslant 2$ and $p\geqslant 2$:
\begin{itemize}
\item[(i)] $\varphi_p^{\rm ab} = M^{(p)}$,
\item[(ii)] $a_r\varphi_p^{-1} = a_r$ and  $a_i\varphi_p^{-1}= a_i(a_{i+1}
    \varphi_p^{-1})^{-p}$ for $i = 1,\ldots,r-1$,
\item[(iii)] $\overline{a_i\varphi_p^{-1}} \in a_iA_r^*a_{i+1}\inv$ for $i=
    1,\ldots,r-1$,
\item[(iv)] $||\varphi_p^{-1}||_1 < 2|a_1\varphi_p^{-1}|$.
\end{itemize}
\end{lem}

\begin{proof}
(i) is clear.

To get (ii), it suffices to compute $(a_i(a_{i+1}\varphi_p^{-1})^{-p})\varphi_p =
(a_i\varphi_p)a_{i+1}^{-p} = a_i$ for $i < r$. Then (iii) follows from (ii) by reverse
induction.

Finally, to see~(iv) observe that by (iii) the product $a_i(a_{i+1}\varphi_p^{-1})^{-p}$ is
reduced and so $|a_i\varphi_p^{-1}| > p|a_{i+1}\varphi_p^{-1}|$ for every $i < r$. Hence
$|a_i\varphi_p^{-1}| < \frac{1}{p^{i-1}}|a_{1}\varphi_p^{-1}|$ for $i = 2,\ldots, r$ and so
 $$
||\varphi_p^{-1}||_1 = \sum_{i=1}^r |a_i\varphi_p^{-1}| < (1+\frac{1}{p} + \cdots +
\frac{1}{p^{r-1}}) |a_1\varphi_p^{-1}| < 2|a_1\varphi_p^{-1}|.  \qedhere
 $$
\end{proof}

Now we are ready to state and prove the lower bounds for our complexity functions.

\begin{thm}\label{highrank}
For every $r\geqslant 2$, there exists constants $K_r,K'_r > 0$ such that, for every $n
\geqslant 1$:
\begin{itemize}
\item[(i)] $K_rn^{r} \leqslant \alpha_r(n)$,
\item[(ii)] $K'_rn^{r-1}\leqslant \beta_r(n)$.
\end{itemize}
\end{thm}

\begin{proof}
Let $p \geqslant r$. By Lemmas \ref{ab}, \ref{invcoin} and \ref{tecnic}(i), we have
\begin{equation}\label{qui1}
\| \varphi_p \|_1 = \| [\varphi_p] \|_1 = \| \varphi_p^{\rm ab} \|_1 = \| M^{(p)} \|_1 = r
+(r-1)p \leqslant rp.
\end{equation}
On the other hand, the same results yield
\begin{equation}\label{qui2}
\| \varphi_p\inv \|_1 \geqslant \| [\varphi_p\inv] \|_1 \geqslant \| (\varphi_p\inv)^{\rm
ab} \|_1 = \| (\varphi_p^{\rm ab})\inv \|_1 = \| (M^{(p)})\inv \|_1 \geqslant p^{r-1}.
\end{equation}

Let $n_0 =\max \left\{ r^2,\, \frac{(r-1)2^{\frac{1}{r-1}}}{2^{\frac{1}{r-1}}-1} \right\}$
and consider $n\geqslant n_0$. Take the integer $p=\lfloor \frac{n}{r}\rfloor \geqslant r$,
which satisfies $\frac{n-(r-1)}{r}\leqslant p\leqslant \frac{n}{r}$ and so $rp\in \{
n-(r-1), \ldots, n \}$. The outer automorphism $[\varphi_p ]\in \Out(F_r)$ satisfies $\|
[\varphi_p] \|_1 \leqslant rp \leqslant n$; and, on the other hand, $\| [\varphi_p\inv]
\|_1 \geqslant p^{r-1} \geqslant (\frac{n-(r-1)}{r})^{r-1} =
\frac{(n-(r-1))^{r-1}}{r^{r-1}}$. Now it is straightforward to check that
 $$
(n-a)^s \geqslant \frac{n^s}{2} \iff n \geqslant \frac{a2^{\frac{1}{s}}}{2^{\frac{1}{s}}-1}
 $$
holds for all positive integers $s,a,n$. Hence, we deduce that
  $$
\| [ \varphi_p\inv ] \|_1 \geqslant \frac{1}{2r^{r-1}}n^{r-1}
  $$
(using that $n\geqslant \frac{(r-1)2^{\frac{1}{r-1}}}{2^{\frac{1}{r-1}}-1}$). We conclude
that $\beta_r(n)\geqslant \frac{1}{2r^{r-1}}n^{r-1}$ for $n\geqslant n_0$. Adjusting the
value of the constant $\frac{1}{2r^{r-1}}$ to cover the finitely many missing values of
$n$, (ii) holds.

To prove (i) let us restrict ourselves to the case $r\geqslant 3$ (Theorem~\ref{lower2}
already deals with the case $r=2$). Fix $p\geqslant r$ and let $\psi_p
=\varphi_p\lambda_{a_1^p}$. Then \eqref{qui1} yields
  $$
\| \psi_p \|_1 = \sum_{i=1}^r |a_1^{-p}(a_i\varphi_p)a_1^p| \leqslant 2rp +\| \varphi_p \|_1
\leqslant 3rp.
 $$
On the other hand,
  $$
\| \psi_p\inv \|_1 = \| \lambda_{a_1^{-p}}\varphi_p\inv \|_1 >\sum_{i=3}^r |(a_1^{p}a_i
a_1^{-p}) \varphi_p\inv|.
 $$
Since the products $(a_1\varphi_p\inv)^p(a_i\varphi_p\inv)(a_1\inv\varphi_p\inv)^p$ are
reduced by Lemma \ref{tecnic}(iii), it follows that $\| \psi_p\inv \|_1 >
2(r-2)p|a_1\varphi_p\inv| >(r-2)p||\varphi_p^{-1}||_1 \geqslant (r-2)p^r$, by Lemma
\ref{tecnic}(iv) and \eqref{qui2}.

This shows that, for $n=3rp$ and $p\geqslant r$, we have $\alpha_r(n) >(r-2)p^r
=\frac{r-2}{(3r)^r}n^r$ i.e., (i) is proven for all such values of $n$. Finally, the
extension of this inequality to all values of $n$ (after adjusting properly the
multiplicative constant) proceeds similarly to part (ii).
\end{proof}

As a final remark for this section, it seems clear that this exhausts the potential of
abelianization techniques to provide lower bounds. If the growths of our complexity
functions are strictly bigger than what we have proven here, this will have to be obtained
by more intricate counting techniques working above the abelian level.

\subsection{Upper bounds}

We can present a polynomial upper bound for $\beta_r(n)$ using Outer space techniques. We
thank M. Bestvina for suggesting a simplification of our initial arguments, which leads to
a very easy and elegant proof of such a polynomial upper bound, now essentially a corollary
of a recent result about the asymmetry of the Lipschitz metric in Outer space.

Let us briefly recall what Outer space $\mathcal{X}_r$ is, $r\geqslant 2$, following the
notation from~\cite{AB} (see~\cite{V} for more details).

By the term \emph{graph} we mean a finite graph $\Gamma$ of rank $r$, all whose vertices
have degree at least three. A {\it metric} on $\Gamma$ is a function $\ell \colon
E\Gamma\to [0,1]$ defined on the set of edges of $\Gamma$ such that $\sum_{e\in E\Gamma}
\ell(e)=1$ and the set of length zero edges forms a forest. Let us denote by
$\Sigma_{\Gamma}$ the space of all such metrics $\ell$ on $\Gamma$, viewed as a ``simplex
with missing faces'' (corresponding to degenerate metrics that vanish on a subgraph which
is not a forest). If $\Gamma'$ is obtained from $\Gamma$ by collapsing a forest, then we
will naturally consider $\Sigma_{\Gamma'}$ as a subset of $\Sigma_{\Gamma}$ along the
inclusion given by assigning length zero to the collapsed edges.

Fix the rose graph $R_r$ with one vertex (denoted $o$) and $r$ edges, and identify the free
group $F_r$ with the fundamental group $\pi_1(R_r, o)$ in such a way that each generator
$a_i$ corresponds to a single oriented edge of $R_r$. Under this identification, each
reduced word in $F_r$ corresponds to a reduced edge-path loop starting and ending at the
basepoint $o$ in $R_r$.

A \emph{marked graph} is a pair $(\Gamma, f)$ where $f$ is a \emph{marking} i.e., a
homotopy equivalence from the rose $R_r$ to $\Gamma$. It is standard to consider the
set of marked graphs modulo the following equivalence relation: $(\Gamma, f) \sim
(\Gamma', f')$ if and only if there is a homeomorphism $\mu \colon \Gamma \to \Gamma'$
such that $f\mu$ is homotopic to $f'$. Denote it by $\mathcal{MG}/\sim$.

Noting that all representatives of a given class $[(\Gamma, f)]\in \mathcal{MG}/\sim$ share
a common underlying graph, we can consider the space of metrics on $\Gamma$ and denote it
$\Sigma_{[(\Gamma, f)]}$. Now, the \emph{Outer Space} $\mathcal{X}_r$ is obtained from the
disjoint union
 $$
\bigsqcup_{[(\Gamma, f)]\in \mathcal{MG}/\sim} \Sigma_{[(\Gamma, f)]}
  $$
by identifying the faces of the simplices along the above natural inclusions. Thus, a point
in $\mathcal{X}_r$ is represented by a triple of the form $(\Gamma, f, \ell)$.

There is a natural action of $\aut F_r$ on $\mathcal{X}_r$. Given $\varphi \in \aut F_r$,
realize it on the rose, say $\varphi\colon R_r \to R_r$, and for every point
$x=(\Gamma,f,\ell)\in \mathcal{X}_r$ define $\varphi \cdot x$ to be $(\Gamma, \varphi f,
\ell)$. It is easy to see that this is well defined and gives an action of $\aut F_r$ on
$\mathcal{X}_r$. Notice that, by construction, inner automorphisms act trivially; so, what
we have is in fact an action of $\Out F_r$ on $\mathcal{X}_r$.

Recently, the Lipschitz metric for $\mathcal{X}_r$ has been introduced and initially
studied in~\cite{FM}, followed by other authors (see, for example, \cite{AB}). This metric
can be defined as follows.

Let $x,x'\in \mathcal{X}_r$ be two points in the Outer space; take representatives, say
$(\Gamma, f,\ell)$ and $(\Gamma', f', \ell')$, respectively. A \emph{difference of
markings} is a map $\mu \colon \Gamma \to \Gamma'$ which is linear on edges, and such that
$f\mu$ is homotopic to $f'$. For such a difference of markings one can define $\sigma(\mu)$
to be the largest slope of $\mu$ over all edges $e\in E\Gamma$. Then define the distance
from $x$ to $x'$ as
  $$
d(x,x')=\min_{\mu} \, \{ \log \sigma(\mu) \},
 $$
where the minimum is taken over all possible differences of markings (and achieved by
Arzela-Ascoli's Theorem).

The basic properties of this ``distance'' are the following: (1) $d(x,y)\geqslant 0$,
with equality if and only if $x=y$; (2) $d(x,z)\leqslant d(x,y)+d(y,z)$ for all
$x,y,z\in \mathcal{X}_r$; (3) $\Out F_r$ acts by isometries i.e., $d([\varphi] \cdot x,
[\varphi]\cdot y)=d(\varphi \cdot x, \varphi\cdot y) = d(x,y)$ for all $x,y\in
\mathcal{X}_r$ and $\varphi \in \aut F_r$; but (4) $d(x,y)\neq d(y,x)$ in general.
See~\cite{FM} and~\cite{AB} for details.

For $\epsilon >0$, define the $\epsilon$-\emph{thick part} of $\mathcal{X}_r$ as
  $$
\mathcal{X}_r (\epsilon) = \{ (\Gamma,f,\ell) \in \mathcal{X}_r \mid \ell(p)\geqslant \epsilon
\,\,\, \forall p \text{ nontrivial closed path in } \Gamma \}.
  $$

The following is an interesting result from Y. Algom-Kfir and M. Bestvina (see
\cite[Theorem~23]{AB}):

\begin{thm}[Algom-Kfir, Bestvina]\label{AB}
Let $r\geqslant 2$. For any $\epsilon>0$ there is a constant $M=M(r, \epsilon) > 0$ such
that, for all $x,y \in \mathcal{X}_r(\epsilon)$,
 $$
d(x,y) \leqslant M \cdot d(y,x).
 $$
\end{thm}

As an easy corollary, we obtain our polynomial upper bound for $\beta_r(n)$:

\begin{cor}\label{highup}
For every $r\geqslant 2$, there exist constants $K_r,M_r > 0$ such that $\beta_r(n)
\leqslant K_rn^{M_r}$ for every $n\geqslant 1$.
\end{cor}

\begin{proof}
Fix an automorphism $\varphi \in \aut F_r$.

Consider the point of the Outer space $x\in \mathcal{X}_r$ represented by the triple
$(R_r, id, \ell_0)$ i.e., by the identity marking over the balanced rose (here,
$\ell_0$ assigns constant length $1/r$ to each petal). Now consider the point
$[\varphi] \cdot x=(R_r,\varphi,\ell_0)\in \mathcal{X}_r$. From the definitions,
$\mu:R_r \to R_r$ is a difference of markings if and only if $\mu$ is homotopic to
$\varphi$; and it is straightforward to see that this happens if and only if $\mu =
\varphi\lambda_w \lambda_p$ for some $w\in F_r$ and some path $p$ travelling linearly
from the basepoint $o$ to an internal point of a petal and with $\ell (p)\leqslant
\frac{1}{2r}$ (if $\mu$ fixes the basepoint then $p$ can be taken to be trivial;
otherwise, it can always be taken to be the shortest path from $o$ to $o\mu$).
Moreover, $\mu$ maps each edge $a_i$ linearly to a path of length
$\ell(p)+|a_i\varphi\lambda_w|\frac{1}{r}+\ell(p)$ so, $\sigma(\mu
)=\sigma(\varphi\lambda_w \lambda_p) = \| \varphi\lambda_w \|_{\infty}+2r\ell(p)$. It
follows that
  $$
\begin{array}{rcl} d(x,[\varphi]\cdot x) & = & \min_{w,\, p} \, \{ \log(\sigma(\varphi
\lambda_w \lambda_p)) \} \\ & = & \log (\min_{w,\, p} \, (\| \varphi\lambda_w \|_{\infty}+
2r\ell(p))) \\ & = & \log (\| [\varphi] \|_{\infty}). \end{array}
 $$
Hence, by property (3) above,
 $$
d([\varphi]\cdot x,x) = d(x,[\varphi^{-1}]\cdot x) = \log (\| [\varphi\inv] \|_{\infty}).
 $$
But, since all the involved points belong to the $(1/r)$-thick part
$\mathcal{X}_r(\frac{1}{r})$, we can take the constant $M_r = M(r,\frac{1}{r})$ from
Theorem~\ref{AB} to get $\log (\| [\varphi\inv] \|_{\infty}) \leqslant M_r \log (\|
[\varphi] \|_{\infty})$ and so, $\| [\varphi\inv] \|_{\infty} \leqslant \| [\varphi]
\|_{\infty}^{M_r}$. Bringing in the constant $C_r =C_{\infty,1,r}$ from
Proposition~\ref{constants}, we obtain
 $$
\| [\varphi ]\inv \|_1 \leqslant C_r\| [\varphi ]\inv \|_{\infty} \leqslant C_r\| [\varphi]
\|_{\infty}^{M_r} \leqslant C_r^{M_r+1}\| [\varphi] \|_1^{M_r}.
 $$
Hence $\beta_r(n)\leqslant K_r n^{M_r}$ holds for $K_r =C_r^{M_r+1}$.
\end{proof}

\begin{rmk}
\emph{Theorems~\ref{highrank}(ii) and Corollary~\ref{highup} bound the gap for outer
automorphism inversion in free groups $F_r$ or rank $r\geqslant 3$ between polynomial
with degree $r-1$ and polynomial with degree $M$ for a big enough $M$. This is all the
information known at the moment about Question~\ref{q}. These two bounds are far from
each other and, intuitively, both of them far from sharp. The proof for the lower bound
uses only information coming from the abelianization so, it seems plausible that,
playing with more sophisticated automorphisms of $F_r$ than the $\varphi_p$'s
constructed above, one could improve the degree of the lower bound. On the other hand,
the proof of Algom-Kfir-Bestvina's theorem is indirect and the actual constant provided
there is quite big, indicating that maybe the degree of the upper bound provided for
$\beta_r(n)$ may also be improved.}
\end{rmk}

\begin{rmk}
\emph{We also remark that getting a polynomial upper bound for $\alpha_r (n)$ seems to be
more complicated (see Question~\ref{w}). On the one hand, the geometric techniques coming
from Outer space do not provide control on the length of possible conjugators showing up
when computing the pre-image of the generators $a_i$ by a (even cyclically reduced) given
automorphism of $F_r$. A possibility here could be to try translating the argument above
from the Outer space to the Auter space concerning real automorphisms (not just outer
ones); unfortunately, the theory for the Auter space is much less developed and, for
example, there is no known metric and so no analog to Algom-Kfir-Bestvina's theorem, yet.
On the other hand, and oppositely to the much easier case $r=2$, these conjugators cannot
be avoided in general by just composing with an appropriate inner automorphism because they
can affect differently the various generators.}
\end{rmk}

\subsection*{Acknowledgements}

We thank Mladen Bestvina, Warren Dicks and Ant\'{o}nio Machiavelo for their valuable
suggestions at different stages of this work.

The first named author gratefully acknowledges partial support from  MICINN (Spain), grant
MTM 2009-14464-C02  (European FEDER support included).

The second named author acknowledges support from the European Regional Development Fund
through the programme COMPETE and from the Portuguese Government through FCT -- Funda\c
c\~ao para a Ci\^encia e a Tecnologia, under the project PEst-C/MAT/UI0144/2011.

The third named author gratefully acknowledges partial support from the MEC (Spain) and the
EFRD (EC) through project numbers MTM2008-01550 and MTM2011-25955.

\end{document}